\documentclass[preprint,11pt]{article}
\usepackage{color}
\usepackage{latexsym}
\usepackage{amsmath}
\usepackage{amsthm}
\usepackage{amssymb}
\usepackage{mathrsfs}
\usepackage{authblk} 
\usepackage{hyperref}
\numberwithin{equation}{section}

\hypersetup{hypertex=true,
 colorlinks=true,
    linkcolor=blue,
         anchorcolor=blue,
           citecolor=blue}
\allowdisplaybreaks
\newtheorem{thm}{Theorem}[section]

\newtheorem{lem}[thm]{Lemma}

\newtheorem{prop}[thm]{Proposition}
\newtheorem{rem}[thm]{Remark}
\def\E{\mathrm{E}}
\def\Var{\mathrm{Var}}

\def\d{\mathrm{d}}
\def\Cov{\mathrm{Cov}}
\def\TV{\mathrm{TV}}
\def\N{\mathrm{N}}
\def\P{\mathrm{P}}
\def\R{\mathbb{R}}

\def\e{\varepsilon}
\def\a{\alpha}

\def\u{\widetilde{u}}

\def\G{\widetilde{G}}

\def\de{\delta}
\begin{document}

\title{\bf{\textmd{Convergence of densities of spatial averages of the linear stochastic heat equation}}}

\author{\small Wanying Zhang}
\author{Yong Zhang\thanks{The corresponding address: zyong2661@jlu.edu.cn}}
\author{Jingyu Li}
\affil{School of Mathematics, Jilin University, Changchun 130012, China}

\date{}

\maketitle

\begin{abstract}
{ Let $\{u(t,x)\}_{t>0,x\in{{\mathbb R}^{d}}}$ denote the solution to the linear (fractional) stochastic heat equation. We establish rates of convergence with respect to the uniform distance between the density of spatial averages of solution and the density of the standard normal distribution in some different scenarios. We first consider the case that $u_0\equiv1$, and the stochastic fractional heat equation is driven by a space-time white noise. When $\a=2$ (parabolic Anderson model, PAM for short) and the stochastic heat equation is driven by colored noise in space, we present the rates of convergence respectively in the case that $u_0\equiv1$, $d\geq1$ and $u_0=\de_0$, $d=1$ under an additional condition $\hat f(\R^d)<\infty$. Our results are obtained by using a combination of the Malliavin calculus and Stein's method for normal approximations.\\}


\end{abstract}

\par
\section{Introduction}  
Consider the following stochastic (fractional) heat equation:
{\begin{equation}\label{1.1}
\left\{
\begin{array}{lr}
{\partial _t}{u(t,x)}={v\cdot(-(-\Delta))^{\frac{\a}{2}} u(t,x)}+{u(t,x)}{\eta (t,x)}~~~{\rm for} ~{(t,x)} \in {(0,+\infty)}{\times{\mathbb R}^{d}},\\
{\rm subject~to}~~~~~{u(0,x)=u_0(x)},\\
\end{array}
\right.
\end{equation}
where $v$ is a positive constant, $-(-\Delta)^{\frac{\a}{2}}$ denotes the fractional Laplace operator (see \cite{G19}) and $\eta$ denotes a centered, generalized Gaussian random field such that
\begin{align}\label{1.2}
{\rm Cov}[\eta(t,x),\eta(s,y)]=\delta_{0}(t-s)f(x-y)~~~{\rm for~all}~ s,t\geq0 ~{\rm and} ~x,y \in {\mathbb R}^{d},
\end{align}
for a non-zero, nonnegative-definite, tempered Borel measure $f$ on ${\mathbb R}^{d}$.

We are interested in the rates of convergence of the uniform norm of densities in the following three cases:\\
{\bf Case} 1 $(d=1)$: $\a\in(1,2]$, $v=1$, $u_0\equiv1$ and $f(x)=\de_0(x)$ for all $x\in\R$.\\
{\bf Case} 2 $(d\geq1)$: $\a=2$, $v=\frac{1}{2}$, $u_0\equiv1$, $f(\R^d)<\infty$ and ${\hat f}(\R^d)<\infty$.\\
{\bf Case} 3 $(d=1)$: $\a=2$, $v=\frac{1}{2}$, $u_0=\delta_0$, $f(\R)<\infty$ and ${\hat f}(\R)<\infty$.

Following from Walsh \cite{Walsh86}, we can interpret (\ref{1.1}) in the following mild form:\\
In case 1,
\begin{align}\label{1.3}
u(t,x)=1+\int_{(0,t)\times{\mathbb R}}G_\a(t-s,x-y)u(s,y)\eta({\rm d}s,{\rm d}y)~~~{\rm for}~t>0~{\rm and}~x\in {\mathbb R},
\end{align}
where $G_\a$ denotes the Green kernel defined through its Fourier transform $\widehat{G_\a(t,\cdot)}(x)=e^{-t\vert x\vert^\a}$.\\
In case 2,
\begin{align}\label{1.4}
u(t,x)=1+\int_{(0,t)\times{\mathbb R}^{d}}{p}_{t-s}(x-y)u(s,y)\eta({\rm d}s,{\rm d}y)~~~{\rm for}~t>0~{\rm and}~x\in {\mathbb R}^{d},
\end{align}
where $p_t(x)$ denotes the heat kernel that satisfies $p_t(x)=\frac{1}{\sqrt{2\pi t}}e^{-\frac{\Vert x\Vert^2}{2t}}$.\\
In case 3,
\begin{equation}\label{1.5}
U(t,x)=1+\int_{(0,t)\times{\mathbb R}}{p}_{s(t-s)/t}\left(y-\frac{s}{t}x\right)U(s,y)\eta({\rm d}s,{\rm d}y)~~~{\rm for}~t>0~{\rm and}~x\in {\mathbb R},
\end{equation}
where $U(t,x):=u(t,x)/p_t(x)$ (see \cite{Davar21}).

The existence and uniqueness problems for the solution to (\ref{1.1}) have been studied extensively \cite{Ass22,Chen15,Chen19,Dalang99,Davar21}. In the present setting, we must ensure that the Fourier transform $\hat f$ satisfies the integrability condition:
\begin{equation}\label{1.6}
\Upsilon(\beta):=\frac{1}{({2\pi})^{d}}\int_{{\mathbb R}^{d}}\frac{{\hat f({\rm d}y)}}{\beta+{\Vert y \Vert}^{\a}}<\infty~~~{\rm for~all}~\beta>0.
\end{equation}
Clearly, (\ref{1.6}) holds in the cases mentioned above.

For any fixed $t>0$ and $N>0$, we introduce the centered and normalized spatial averages:\\
In case 1,
\begin{equation}\label{1.7}
F_{N,1}:=\frac{1}{\sigma_{N,1}}\left(\int_{[0,N]}\left(u(t,x)-1\right)\d x\right),~{\rm where}~{\sigma^{2}_{N,1}}:=\Var\left(\int_{[0,N]} u(t,x)\d x\right).
\end{equation}
In case 2,
\begin{equation}\label{1.8}
F_{N,2}:=\frac{1}{\sigma_{N,2}}\left(\int_{[0,N]^d}\left(u(t,x)-1\right)\d x\right),~{\rm where}~{\sigma^2_{N,2}}:=\Var\left(\int_{[0,N]^d} u(t,x)\d x\right).
\end{equation}
In case 3,
\begin{equation}\label{1.9}
F_{N,3}:=\frac{1}{\sigma_{N,3}}\left(\int_{[0,N]}\left(U(t,x)-1\right)\d x\right),~{\rm where}~{\sigma^2_{N,3}}:=\Var\left(\int_{[0,N]} U(t,x)\d x\right).
\end{equation}

There are many arguments on quantitative central limit theorem (CLT for short) for spatial averages of the solution to (\ref{1.1}). CLT using techniques of Malliavin-Stein's method was first derived by Huang et al \cite{Huang20} with $\a=2$, $u_0\equiv1$ and $f=\delta_0$. Later, Chen et al deduced the case that $d\geq1$ in \cite{Chen22} and \cite{Chen23} under the condition $f(\R^d)<\infty$. As for delta initial condition, Chen et al \cite{Chen222} proved the CLT for the PAM when $\eta$ is a space-time white noise. After that, Khoshnevisan et al \cite{Davar21} extended the result to cover the condition that $\eta$ is colored in space. Recently, Assaad et al \cite{Ass22} studied the case of stochastic fractional heat equation with initial condition $u_0\equiv1$. Other related limit theorems and its variations can be found in \cite{Huang220,{Hu19},Li21,Li23,Nualart21, Nualart22,Zhang22}. In the present setting, we can express the convergence rates for the total variation distance\footnote{The total variation distance between two random variables $X$ and $Y$ is defined as $d_{\TV}(X,Y)=\sup_{B\in {\mathcal B}(\R)}\vert {\rm P}(X\in B)-{\rm P}(Y\in B)\vert$, where $\mathcal{B}(\R)$ is the collection of all Borel sets in $\R$.} precisely as the following:
\begin{align}
d_{\TV}\left(F_{N,1}, {\N}(0,1)\right)\leq\frac{C_t}{\sqrt N}~~~{\rm for~all}~N\geq1,\label{1.10}
\end{align}
\begin{align}
d_{\rm TV}\left(F_{N,2}, {\rm N}(0,1)\right)\leq\frac{C_t}{(\sqrt N)^d}~~~{\rm for~all}~N\geq 1,\label{1.11}
\end{align}
and
\begin{equation}\label{1.12}
d_{\rm TV}\left(F_{N,3}, {\rm N}(0,1)\right)\leq\frac{C_t\sqrt {\log N}}{\sqrt N}~~~{\rm for~all}~N\geq e.
\end{equation}

The above results describe the convergence rates for total variation distance of distributions. In this paper, we obtain the convergence rates for the uniform distance of densities in the mentioned three cases. Analyzing upper bounds for the uniform norm of densities by using Malliavin calculus techniques was first introduced by Hu et al \cite{Hu14}. Recently, Kuzgun and Nualart \cite{Ku22} established the upper bounds for uniform distance of densities between a given random variable $F$ and the standard normal random variable. Moreover, when $f=\de_0$, they deduced the convergence rates with respect to the nonlinear stochastic heat equation and the PAM under the condition $u_0\equiv1$ and $u_0=\de_0$, respectively.
Inspired by these ideas, we obtain the following results:
\begin{thm}\label{thm1.1}
In case 1, let $F_{N,1}$ be the spatial average defined in (\ref{1.7}). Then, for all $N\geq1$,
\begin{equation}\label{1.14}
\sup_{x\in\R}\vert f_{F_{N,1}}(x)-\phi(x)\vert\leq\frac{C_t}{\sqrt N},
\end{equation}
where $f_{F_{N,1}}$ and $\phi$ denote the densities of ${F_{N,1}}$ and ${\rm N}(0,1)$, respectively.
\end{thm}
One can see that (\ref{1.1}) becomes the PAM when $\a=2$. Hence, Theorem \ref{thm1.1} is an extension of the linear case in \cite[Theorem 1.1]{Ku22}. However, because the Green kernel $G_\a(x)$ is defined by a Fourier transform, technical computations are much more complex in proving (\ref{1.14}). The preceding theorem is given for $f=\de_0$. In the following, we concentrate on the case that $\eta$ is colored in space.

\begin{thm}\label{thm1.2}
In case 2, let $F_{N,2}$ be the spatial average defined in (\ref{1.8}). Then, for all $N\geq1$,
\begin{equation}\label{1.15}
\sup_{x\in\R}\vert f_{F_{N,2}}(x)-\phi(x)\vert\leq\frac{C_t}{(\sqrt N)^d},
\end{equation}
where $f_{F_{N,2}}$ and $\phi$ denote the densities of $F_{N,2}$ and ${\rm N}(0,1)$, respectively.
\end{thm}

\begin{thm}\label{thm1.3}
In case 3, let $F_{N,3}$ be the spatial average defined in (\ref{1.9}). Fix $\beta>21$. Then, for all $N\geq e$,
\begin{equation}\label{1.15}
\sup_{x\in\R}\vert f_{F_{N,3}}(x)-\phi(x)\vert\leq\frac{C_t(\sqrt {\log N})^{\beta}}{\sqrt N},
\end{equation}
where $f_{F_{N,3}}$ and $\phi$ denote the densities of $F_{N,3}$ and ${\rm N}(0,1)$, respectively.
\end{thm}
According to the method in \cite{Ku22} (see Lemma \ref{lem2.1}), the estimates on $p$-norm of the second Malliavin derivative and $\Vert(D_vF)^{-1}\Vert_4$ are fundamental ingredients in discussing the convergence rates. Notice that, when dealing with $\Vert(D_vF)^{-1}\Vert_4$, we must ensure that negative moment of the solution $u(t,x)$ exists. In the case of $d=1$, $f=\de_0$, Chen et al \cite{Chen21} have proved that $\E\left[(u(t,x))^{-p}\right]<\infty~{\rm for~any~fixed}~(t,x)\in(0,+\infty)\times\R$ based on small-ball probability estimate. When $\eta$ is colored in space, small-ball probability estimate was given by Chen and Huang \cite{Chen19}. However, the negative moment of the solution exists only if $\hat f(\R^d)<\infty$ (see Lemma \ref{lem2.3} part (4) and Lemma \ref{lem2.4} part (4)).
\begin{rem}
The collection of measures under the condition $\hat f(\R^d)<\infty$ is massive. For example, $f$ is given by a Gaussian kernel ($f(\d x)={p}_1(x)\d x$ and ${\hat f}(\d x)=e^{-\Vert x \Vert^2/2}\d x$) or a Cauchy kernel ($f(\d x)=\left[\prod_{j=1}^d(1+\vert x_j \vert^2)\right]^{-1}\d x$ and ${\hat f}(\d x)=\pi^d\prod_{j=1}^d e^{-\vert x_j \vert}\d x$).
\end{rem}
\begin{rem}
Unfortunately, the multidimensional situation of case 3 has not been investigated so far. It has been proved in \cite[Theorem 1.3]{Davar21} that the convergence rate for CLT in terms of total variation is $N^{-1/2}$, which implies that the convergence rate of $\Vert 1-D_vF_N\Vert_2$ in Lemma \ref{lem2.1} is $N^{-1/2}$. However, $\Vert(D_vF_N)^{-1}\Vert_4$ can not be controlled by $N^{a}$ for any $a<1/2$ through our method currently. This is an issue for future research to explore. 
\end{rem}
The organization of this paper is as follows. In Section 3, we derive moment estimates of the second Malliavin derivative of $u(t,x)$. Specifically, we obtain a more general result for the nonlinear stochastic fractional heat equation in case 1. Section 4 is devoted to analyzing the negative moments of $D_vF$. Moreover, we prove the convergence rates on uniform distance in Theorems \ref{thm1.1}-\ref{thm1.3} in Section 5, based on Malliavin-Stein's method and Fourier analysis. Finally, in the Appendix, we introduce some technical lemmas that are used along the paper.

Throughout this paper, we write $\Vert Z\Vert_p$ instead of $(\E \vert Z\vert^p)^{1/p}$ for any $Z\in L^p$ and we denote the generic nonnegative constant by $C$ that might take different values and depend on different variables.
\section{Preliminaries}
\subsection{The BDG inequality}

For every continuous $L^{2}(\Omega)$-martingale $\{M_{t}\}_{t\geq0}$, we have the following Burkholder-Davis-Gundy (BDG for short) inequality:
$$
\E\left({\vert M_{t}\vert}^{k}\right)\leq z_{k}^{k}\E{\left({\langle M\rangle}_{t}^{k/2}\right)}~~~{\rm for}~{\rm all}~t\geq 0~{\rm and}~k\geq2,
$$
where $\{z_{k}\}_{k\geq2}$ are the optimal constants. Moreover, the method in \cite{Carlen91} and \cite{Davis76} together implies that
$$
z_{2}=1,~~~{\rm and}~~~\sup_{k\geq2}\frac{z_{k}}{\sqrt{k}}=\lim_{k\rightarrow \infty}\frac{z_{k}}{\sqrt {k}}=2,
$$
which means ${z_{k}}$ is bounded from above by the multiples of $\sqrt{k}$, uniformly for all $k\geq2$.

\subsection{Malliavin calculus and Stein's method}
Denote by ${\mathcal H}_{0}$ the reproducing kernel Hilbert space spanned by all real-valued functions on ${\mathbb R}^{d}$, with respect to the scalar product $\langle\phi,\psi\rangle:={\langle\phi,{\psi\ast f\rangle}}_{L^{2}({\mathbb R}^{d})}$, and let ${\mathcal H}=L^{2}({\mathbb R}_{+}\times{\mathcal H}_{0})$. The Gaussian random field $\{W(h)\}_{h\in{\mathcal H}}$ formed by such Wiener integrals
\begin{equation}\label{2.1}
W(h):=\int_{{\mathbb R}^{+}\times{\mathbb R}^{d}}h(s,y)\eta({{\rm d}s,{\rm d}y})
\end{equation}
defines an isonormal Gaussian process on the Hilbert space ${\mathcal H}$. On the basis of this, we can develop the Malliavin calculus (see, for instance, \cite{Nualart06}).

Let ${\mathcal S}$ denote the space of simple random variables of the form
$$
F=f(W(h_1,\ldots,W(h_d))),
$$
where $f\in C_p^\infty(\R^d)$, that is, $f$ is a smooth function and all its partial derivatives have at most polynomial growth at infinity, and $h_1,\ldots,h_n\in {\mathcal H}$. Then the Malliavin derivative $DF$ is defined as $\mathcal H$-valued random variable
\begin{equation}
DF=\sum_{i=1}^d\frac{\partial f}{\partial x_i}\left(W(h_1),\ldots,W(h_d)\right)h_i.
\end{equation}
For any $p\geq1$ and $k\geq1$, we denote the completion of ${\mathcal S}$ by ${\mathbb D}^{k,p}$ with respect to the norm
$$
\Vert F\Vert_{k,p}=\left(\E\left[\vert F\vert^p\right]+\sum_{j=1}^k\E\left[\Vert D^j F\Vert_{{\mathcal H}^{\otimes j}}^p\right]\right)^{1/p},
$$
where $D^j$ denotes the $j$-th iterated Malliavin derivative and $\otimes$ denotes the tensor product. The adjoint operator $\de$ of the derivative $D$ is characterized by the duality formula
$$
\E[F\de(v)]=E[\langle DF,v\rangle_{\mathcal H}],
$$
for any $F\in {\mathcal H}$. An important property of $\de$ is that any predictable and square-integrable random field $v$ belongs to the domain of $\de$ and $\de(v)$ coincides with the Walsh integral, that is,
$$
\delta(v)=\int_{{\mathbb R}^{+}\times{\mathbb R}^{d}}v(s,x)\eta({{\rm d}s,{\rm d}x}).
$$

For an ${\mathcal H}$-valued random variable $v$ and $F\in {\mathbb D}^{1,1}$, define
\begin{equation}
D_vF:=\langle DF,v\rangle_{\mathcal H}.
\end{equation}
The following lemma, which characterizes the uniform distance of densities, plays an important role in proving Theorems \ref{thm1.1}-\ref{thm1.3}.
\begin{lem}\label{lem2.1}{\rm\cite[Theorem 3.2]{Ku22}}
Assume that $v\in{\mathbb D}^{1,6}({\mathcal H})$, $F=\de(v)\in{\mathbb D}^{2,6}$ with $\E(F)=0$ and $\E(F^2)=1$, and $\Vert (D_vF)^{-1}\Vert_4<\infty$. Then $F$ admits a density $f_F(x)$ and the following inequality holds true,
\begin{align}
\sup_{x\in\R}\vert f_F(x)-\phi(x)\vert &\leq\left(\Vert F\Vert_4\Vert (D_vF)^{-1}\Vert_4+2\right)\Vert 1-D_vF\Vert_2\notag\\
&~~+\Vert (D_vF)^{-1}\Vert_4^2\Vert D_v(D_vF)\Vert_2\label{2.4},
\end{align}
where $\phi(x)$ denotes the density of $\N(0,1)$.
\end{lem}
\subsection{Some properties of $u(t,x)$}
We first introduce some properties of the moments and Malliavin derivative of $u(t,x)$ in the following three lemmas.
\begin{lem}\label{lem2.2}
Let $u(t,x)$ be the solution to (\ref{1.1}) in case 1, we have\\
$(1)$ {\rm\cite[Theorem 2.1]{Ass22}} The process $u(t):=\{u(t,x)\}_{x\in\R}$ is stationary. Moreover, for any $p\geq1$ and any $T>0$,
\begin{equation*}
\sup_{t\in(0,T],x\in\R}\E[\vert u(t,x)\vert^p]<\infty.
\end{equation*}
$(2)$ {\rm\cite[Propositions 5.1-5.2]{Ass22}} For almost all $0<s<t<T$, $x,y\in\R$,
\begin{equation}
D_{s,y}u(t,x)= G_{\alpha}(t-s, x-y)u(s,y)+\int_{s}^{t} \int_{\mathbb{R}} G_{\alpha}(t-r, x-z) D_{s, y} u(r, z) \eta(\d r, \d z).\label{2.5}
\end{equation}Moreover, for all $p\geq2$,
\begin{equation*}
\Vert D_{s,y}u(t,x)\Vert_p\leq C_{T,p}(t-s)^{-\frac{1}{2\a}}G_\a^{\frac{1}{2}}(t-s,x-y).
\end{equation*}
$(3)$ {\rm\cite[Theorem 1.5]{Chen21}} Fix $(t,x)\in\R^+\times\R$, for all $p>0$,
\begin{equation*}
\E\left([u(t,x)]^{-p}\right)<\infty.
\end{equation*}
\end{lem}
\begin{lem}\label{lem2.3}
Let $u(t,x)$ be the solution to (\ref{1.1}) in case 2, we have\\
$(1)$ {\rm\cite[Theorem 1.1]{Chen221}} The random field $u(t):=\{u(t,x)\}_{x\in\R^d}$ is stationary.\\
$(2)$ {\rm\cite[Lemma 4.2]{Chen22}} For almost all $0<s<t<T$, $x,y\in\R^d$,
\begin{equation}
D_{s,y}u(t,x)= p_{t-s}(x-y)u(s,y)+\int_{s}^{t} \int_{\mathbb{R}^d} p_{t-r}(x-z) D_{s, y} u(r, z) \eta(\d r, \d z).\label{2.6}
\end{equation}
Moreover, for all $p\geq2$,
\begin{equation*}
\sup_{t\in(0,T],x\in\R^d}\E[\vert u(t,x)\vert^p]<\infty~{\it and}~
\Vert D_{s,y}u(t,x)\Vert_p\leq C_{T,p}p_{t-s}(x-y).
\end{equation*}
$(3)$ {\rm\cite[Proposition 3.4]{Chen23}} For all $0<r<s<t$, $x,y,z\in\R^d$ and for every $p\geq2$,
\begin{equation*}
\Vert D_{r,z}D_{s,y}u(t,x)\Vert_p\leq C_{t,p}p_{t-s}(x-y)p_{s-r}(y-z).
\end{equation*}
$(4)$ Fix $(t,x)\in\R^+\times\R^d$, for all $p>0$,
\begin{equation*}
\E\left([u(t,x)]^{-p}\right)<\infty.
\end{equation*}
\end{lem}
\begin{proof}
(4). From \cite[Theorem 1.6]{Chen19}, since $\hat{f}(\R^d)<\infty$, we have that, for any fixed $(t,x)\in (0,+\infty)\times \R^d$, there exists a finite constant $A=A_{t,x}>0$ such that for all $\e>0$ small enough,
\begin{align}
\P(u(t,x)<\e)\leq A\exp(-A|\log\e|(\log|\log\e|)^{2}).\label{2.7}
\end{align}
According to \cite[Lemma A.1]{Chen21}, it suffices to show that, for all $p>0$, there exists some finite constant $C_p>0$ such that
\begin{align}
\P(u(t,x)<\e)\leq C_p\e^p.\label{2.8}
\end{align}
For any $p>0$, choose $C_p=A\vee\exp(p\exp({\sqrt{p/A}}))$. Then, for any $0<\e<\exp(-\exp(\sqrt{p/A}))$,
\begin{align}
\P(u(t,x)<\e)\leq A\exp(-A|\log\e|(\log|\log\e|)^{2})\leq C_p\e^p, \label{2.9}
\end{align}
and for any $\e\geq\exp(-\exp(\sqrt{p/A}))$,
\begin{align}
\P(u(t,x)<\e)\leq1\leq \exp(p\exp(\sqrt{p/A})) \e^p\leq C_p \e^p.\label{2.10}
\end{align}
Combining (\ref{2.9}) and (\ref{2.10}), we prove the result.
\end{proof}
\begin{lem}\label{lem2.4}
Let $u(t,x)$ be the solution to (\ref{1.1}) in case 3, $U(t,x)=u(t,x)/p_t(x)$, we have\\
$(1)$ {\rm\cite[Theorem 1.1]{Davar21}} The random field $U(t):=\{u(t,x)\}_{x\in\R}$ is stationary. Moreover, for any $p\geq2$ and any $T>0$,\\
\begin{equation*}
\sup_{t\in(0,T],x\in\R}\E[\vert U(t,x)\vert^p]<\infty.
\end{equation*}
$(2)$ {\rm\cite[Proposition 4.1]{Davar21}} For all $0<r<s<t$, $x,y,z\in\R$,
\begin{equation}
D_{s,y}U(t,x)= p_{\frac{s(t-s)}{t}}\left(y-\frac{s}{t}x\right)U(s,y)+\int_{s}^{t} \int_{\mathbb{R}} p_{\frac{r(t-r)}{t}}\left(z-\frac{r}{t}x\right) D_{s, y} U(r, z) \eta(\d r, \d z).\label{2.11}
\end{equation}
$(3)$ {\rm\cite[Corallary 1.2]{Ku20}} Suppose that $0<r_i<r_j<t$ for all $1\leq i<j\leq k$ and $z_i\in\R$ for $1\leq i\leq k$, let $D^k_{r_k,z_k}U(t,x)$ denote the $k$-th iterated Malliavin derivative of $U(t, x)$, i.e., $D^k_{r_k,z_k}U(t,x)=D_{r_1,z_1}\ldots D_{r_k,z_k}U(t,x)$, then for all $p\geq2$,
\begin{equation*}
\Vert D^k_{r_k,z_k}U(t,x)\Vert_p\leq C_{t,p}\left(\prod_{m=1}^{k-1}p_{\frac{{r_m(r_{m+1}-r_m)}}{r_{m+1}}}\left(z_{m}-\frac{r_m}{r_{m+1}}z_{m+1}\right)\right)
p_{\frac{{r_k(t-r_k)}}{t}}\left(z_{k}-\frac{r_k}{t}x\right).
\end{equation*}
$(4)$ Fix $(t,x)\in\R^+\times\R$, for all $p>0$,
\begin{equation*}
\E\left([u(t,x)]^{-p}\right)<\infty,~hence,~\E\left([U(t,x)]^{-p}\right)<\infty.
\end{equation*}
\end{lem}
\begin{proof}
(4). The argument is the same as the proof of Lemma \ref{lem2.3} part (4).
\end{proof}

The following lemma states the asymptotic behavior of the variance functions and upper bounds of the moment of spatial averages.
\begin{lem}\label{lem2.5}
Let $F_{N,i}$ and $\sigma^2_{N,i}$ $(i=1,2,3.)$  be as defined in (\ref{1.7})-(\ref{1.9}). Then, we have\\
$(1)$ For any $p\geq2$, $\sup_{N\geq1}\|F_{N,i}\|_p\leq C_t$ {\rm($i=1,2$)}, and $\sup_{N\geq e}\|F_{N,3}\|_p\leq C_t$.\\
$(2)$ {\rm\cite[Theorem 5.6]{Ass22}} $\lim_{N\rightarrow\infty}\frac{\sigma^2_{N,1}}{N}=t.$\\
$(3)$ {\rm\cite[Proposition 5.2]{Chen23}} $\lim_{N\rightarrow\infty}\frac{\sigma^2_{N,2}}{N^d}=\int_{\R^d}\Cov[u(t,x),u(t,0)]\d x<\infty.$\\
$(4)$ {\rm\cite[Theorems 5.1-5.2]{Davar21}}
$\lim_{N\rightarrow\infty}\frac{\sigma^2_{N,3}}{N\log N}=tf(\R)$.
\end{lem}
\begin{proof}(1). The upper bounds of the moments follow easily from the BDG inequality; see, for instance, \cite[Lemma 2.4]{Chen222}.
\end{proof}
\section{Second Malliavin derivative}
In this section, we will estimate the moment bounds of the second Malliavin derivative of $u(t,x)$. The estimates in case 2 and 3 can be found in Lemmas \ref{lem2.3}-\ref{lem2.4}. Hence, we prove the result only in case 1. Consider the more general setting,
{\begin{equation}\label{3.1}
\left\{
\begin{array}{lr}
{\partial _t}{u(t,x)}={-(-\Delta)^{\frac{\a}{2}} u(t,x)}+\sigma({u(t,x)}){\eta (t,x)}~~~{\rm for} ~{(t,x)} \in {(0,+\infty)}{\times{\mathbb R}},\\
{\rm subject~to}~~~~~{u(0,x)=1},\\
\end{array}
\right.
\end{equation}
where $\sigma$ denotes the Lipschitz function satisfied $\sigma(1)\neq0$. In order to obtain the following result, we further assume that $\sigma$ is twice continuously differentiable, $\sigma'$ is bounded and $\vert \sigma''(x)\vert\leq C(1+\vert x\vert^m)$ for some $m>0$.

\begin{prop}\label{prop3.1}Let $u(t,x)$ denote the solution to (\ref{3.1}). Then, $u(t, x) \in \cap_{p \geq 2} \mathbb{D}^{2, p}$ and for almost all $0<r<s<t$, $y, z \in \mathbb{R}$, we have
\begin{equation*}
\begin{aligned}
D_{r, z}D_{s, y}u(t, x)=& G_{\alpha}(t-s,x-y) \sigma^{\prime}(u(s, y))D_{r, z}u(s, y)\\
&+\int_{[s, t] \times \mathbb{R}} G_{\alpha}(t-\tau,x-\xi) \sigma^{\prime \prime}(u(\tau, \xi)) D_{r, z} u(\tau, \xi) D_{s, y} u(\tau, \xi) \eta(\d \tau, \d \xi) \\
&+\int_{[s, t] \times \mathbb{R}} G_{\alpha}(t-\tau,x-\xi) \sigma^{\prime}(u(\tau, \xi)) D_{r, z} D_{s, y} u(\tau, \xi) \eta(\d \tau, \d \xi) .
\end{aligned}
\end{equation*}
Moreover, for all $0 \leq r<s<t \leq T$ and $x, y, z \in \mathbb{R}$,
\begin{align}
\left\|D_{r, z} D_{s, y} u(t, x)\right\|^2_{p} \leq C_{T, p} K^2_{r, z, s, y}(t, x),\label{3.2}
\end{align}
where
\begin{equation*}
\begin{aligned}
&K^2_{r, z, s, y}(t, x)= (t-s)^{-\frac{1}{\alpha}}G_{\alpha}(t-s,x-y)(s-r)^{-\frac{1}{\alpha}}G_{\alpha}(s-r,y-z)\\
+&\int_s^t\int_{\R}\left[(t-\theta)(\theta-r)(\theta-s)\right]^{-\frac{1}{\alpha}}G_{\alpha}(t-\theta,x-w)
G_{\alpha}(\theta-r,w-z)
G_{\alpha}(\theta-s,w-y)\d \theta\d w.
\end{aligned}
\end{equation*}
In particular, if $\sigma(x)=x$, then
\begin{equation}
\left\|D_{r, z} D_{s, y} u(t, x)\right\|_{p} \leq C_{T, p}(t-s)^{-\frac{1}{2\alpha}}G^{\frac{1}{2}}_{\alpha}(t-s,x-y)(s-r)^{-\frac{1}{2\alpha}}G^{\frac{1}{2}}_{\alpha}(s-r,y-z).\label{3.3}
\end{equation}
\end{prop}
\begin{proof}
First, we define the Picard iteration for the solution to (\ref{3.1}). Let $u_0(t,x)=1$, and for $n\in{\mathbb N}$,
\begin{align}
u_{n+1}(t,x)=1+\int_{(0,t)\times{\mathbb R}}G_\a(t-s,x-y)\sigma(u_n(s,y))\eta({\rm d}s,{\rm d}y)~~~{\rm for}~t>0~{\rm and}~x\in {\mathbb R},\label{3.4}
\end{align}
Applying the properties of the divergence operator \cite[Proposition 1.3.8]{Nualart06}, we deduce, for almost all $(s,y)\in(0,t)\times\R$, that
\begin{align}
D_{s, y} u_{n+1}(t, x)=&G_\a(t-s,x-y)\sigma(u_n(s,y))\notag\\
&+\int_{[s, t] \times \mathbb{R}} G_\a(t-\tau,x-\xi) \sigma'(u_n(\tau,\xi)) D_{s, y} u_{n}(\tau, \xi) \eta(\d \tau, \d \xi),\label{3.5}
\end{align}
and for almost all $s>t$, $D_{s, y} u_{n+1}(t, x)=0$. Using the properties of the divergence operator again, we obtain, for almost all $0<r<s<t$ and $y,z\in\R$,
\begin{align}
D_{r, z} D_{s, y} & u_{n+1}(t, x)=G_\a(t-s,x-y)  \sigma'(u_n(s,y)) D_{r, z} u_{n}(s, y)\notag \\
& +\int_{[s, t] \times \mathbb{R}} G_\a(t-\tau,x-\xi)\sigma''(u_n(\tau,\xi)) D_{r, z} u_{n}(\tau, \xi) D_{s, y} u_{n}(\tau, \xi) \eta(\d \tau, \d \xi)\notag \\
& +\int_{[s, t] \times \mathbb{R}}G_\a(t-\tau,x-\xi) \sigma'(u_n(s,y)) D_{r, z} D_{s, y} u_{n}(\tau, \xi) \eta(\d \tau, \d \xi).\label{3.6}
\end{align}
Moreover, applying the BDG inequality, Minkowski's inequality and Lemma \ref{lem2.2} (Lemma 2.2 still holds for the nonlinear case, see \cite{Ass22}), we have, for almost all $(t,x)\in[0,T]\times\R$,
\begin{align}
&\left\|D_{r, z} D_{s, y} u_{n+1}(t, x)\right\|_{p}^{2}  \leq C_{T, p}G^2_\a(t-s,x-y)(s-r)^{-\frac{1}{\alpha}}G_{\alpha}(s-r,y-z)\notag \\
&~+ C_{T, p} \int_{s}^{t} \int_{\mathbb{R}}
G^2_\a(t-\tau,x-\xi)(\tau-r)^{-\frac{1}{\alpha}}G_{\alpha}(\tau-r,\xi-z)
(\tau-s)^{-\frac{1}{\alpha}}G_{\alpha}(\tau-s,\xi-y)\d \xi \d \tau \notag\\
&~+ C_{T, p} \int_{s}^{t} \int_{\mathbb{R}}G^2_\a(t-\tau,x-\xi)\left\|D_{r, z} D_{s, y} u_{n}(\tau, \xi)\right\|_{p}^{2} \d \xi \d \tau.\label{3.7}
\end{align}
In order to simplify the expression, we define the measure on $[s,t]\times\R$ such that
\begin{equation*}
\begin{aligned}
J(\d \tau,\d \xi):=&(\tau-r)^{-\frac{1}{\alpha}}G_{\alpha}(\tau-r,\xi-z)\delta_{s,y}(\d \tau,\d \xi)\\&+(\tau-r)^{-\frac{1}{\alpha}}G_{\alpha}(\tau-r,\xi-z)
(\tau-s)^{-\frac{1}{\alpha}}G_{\alpha}(\tau-s,\xi-y)\d \tau\d\xi.
\end{aligned}
\end{equation*}
Then, we can rewrite the inequality (\ref{3.7}) as follows:
\begin{equation*}
\begin{aligned}
\left\|D_{r, z} D_{s, y} u_{n+1}(t, x)\right\|_{p}^{2}\leq&
 C_{T, p}\int_{s}^{t} \int_{\mathbb{R}}G^2_\a(t-\tau,x-\xi)J(\d \tau,\d \xi) \\
&+ C_{T, p} \int_{s}^{t} \int_{\mathbb{R}}G^2_\a(t-\tau,x-\xi)\left\|D_{r, z} D_{s, y} u_{n}(\tau, \xi)\right\|_{p}^{2} \d \xi \d \tau.
\end{aligned}
\end{equation*}
Notice that $\left\|D_{r, z} D_{s, y} u_{1}(t, x)\right\|^2_p=0$, then we perform $n-1$ iterations to obtain that
\begin{equation*}
\begin{aligned}
&\left\|D_{r, z} D_{s, y} u_{n+1}(t, x)\right\|_{p}^{2} \leq C_{T, p} \int_{s}^{t} \int_{\mathbb{R}} G^2_\a(t-s_1,x-y_1)J\left(\d s_{1}, \d y_{1}\right) \\
&~+\sum_{k=1}^{n-1} C_{T, p}^{k+1} \int_{s}^{t} \int_{\mathbb{R}} \int_{s}^{s_{1}} \int_{\mathbb{R}} \cdots \int_{s}^{s_{k}} \int_{\mathbb{R}} G^2_\a(t-s_1,x-y_1)G^2_\a(s_1-s_2,y_1-y_2)\cdots\\
&~~~\times
G^2_\a(s_k-s_{k+1},y_k-y_{k+1})J\left(\d s_{k+1}, \d y_{k+1}\right) \d y_{k} \d s_{k} \cdots \d y_{1} \d s_{1} .
\end{aligned}
\end{equation*}
Let $K^2(t,x)$ be defined by $K^2(t,x)=K^2_{r,z,s,y}(t,x):=\int_{s}^t\int_{\mathbb{R}}(t-\tau)^{-\frac{1}{\a}}G _\a(t-\tau,x-\xi) J(\d \tau , \d \xi)$. Using Lemma \ref{lem6.1} part (1) we can write
\begin{equation*}
\begin{aligned}
&\left\|D_{r, z} D_{s, y} u_{n+1}(t, x)\right\|_{p}^{2} \leq C_{T, p} K^2(t,x)
+\sum_{k=1}^{n-1}{C_{T, p}^{k+1}}\int_{s}^{t} \int_{\mathbb{R}} \int_{s}^{s_{1}} \int_{\mathbb{R}} \cdots \int_{s}^{s_{k}} \int_{\mathbb{R}}\\
&~~~\times\prod_{j=0}^{k}\left(s_{j}-s_{j+1}\right)^{-\frac{1}{\a}} G_\a(s_j-s_{j+1},y_j-y_{j+1})J\left(\d s_{k+1}, \d y_{k+1}\right)\d y_{k} \d s_{k} \cdots \d y_{1} \d s_{1},
\end{aligned}
\end{equation*}
where $s_0=t$ and $y_0=x$. Then, apply the semigroup property of the Green kernel to find that
\begin{equation*}
\begin{aligned}
&\left\|D_{r, z} D_{s, y} u_{n+1}(t, x)\right\|_{p}^{2}\leq C_{T, p} K^2(t,x)
+\sum_{k=1}^{n-1}{C_{T, p}^{k+1}}\int_{s}^{t} \int_{\mathbb{R}} \int_{s}^{s_{1}} \cdots \int_{s}^{s_{k}}\\
&~~~\times [(t-s_1)(s_1-s_2)\cdots(s_k-s_{k+1})]^{-\frac{1}{\a}}G_\a(t-s_{k+1},x-y_{k+1})J\left(\d s_{k+1}, \d y_{k+1}\right)\d s_{k}\cdots\d s_{1}\\
&=C_{T, p} K^{2}(t, x)+\sum_{k=1}^{n-1}\left[{C_{T, p}^{k+1}}\int_{\mathbb{R}} \int_{s}^t(t-s_{k+1})^{\frac{(\a-1)k-1}{\a}}G_\a(t-s_{k+1},x-y_{k+1}) J(\d s_{k+1} , \d y_{k+1})\right. \\
&~~~\times \left.\int_{0<r_{k}<\cdots<r_{2}<r_{1}<1}\left[\left(1-r_{1}\right)\left(r_{1}-r_{2}\right) \cdots r_{k}\right]^{-\frac{1}{\a}}\d r_{k} \cdots \d r_{1}\right] \\
&\leq C_{T, p} K^{2}(t, x)+\sum_{k=1}^{n-1}C_{T, p}^{k+1}\frac{T^{\frac{(\a-1)k}{\a}}\Gamma\left(1-\frac{1}{\a}\right)^{k+1}}
{\Gamma\left((k+1)\left(1-\frac{1}{\a}\right)\right)}
\int_{\mathbb{R}}\int_{s}^t(t-\tau)^{-\frac{1}{\a}}G_\a(t-\tau,x-\xi) J(\d \tau , \d \xi)\\
&\leq\left(C_{T,p}+\sum_{k=1}^{\infty}C_{T, p}^{k+1}\frac{T^{\frac{(\a-1)k}{\a}}\Gamma\left(1-\frac{1}{\a}\right)^{k+1}}
{\Gamma\left((k+1)\left(1-\frac{1}{\a}\right)\right)}\right)\int_{\mathbb{R}}\int_{s}^t(t-\tau)^{-\frac{1}{\a}}G_\a(t-\tau,x-\xi) J(\d \tau , \d \xi)\\
&\leq C_{T,p}K^2(t,x),
\end{aligned}
\end{equation*}
where in the second inequality, we use the following identity:
\begin{align}
\int_{0<r_{k}<\cdots<r_{2}<r_{1}<1}\left[\left(1-r_{1}\right)\left(r_{1}-r_{2}\right) \cdots r_{k}\right]^{-\frac{1}{\a}}\d r_{k} \cdots \d r_{1}=\frac{\Gamma\left(1-\frac{1}{\a}\right)^{k+1}}
{\Gamma\left((k+1)\left(1-\frac{1}{\a}\right)\right)}.\label{3.8}
\end{align}
Hence, we obtain the moment estimate
\begin{align}\sup _{n \in \mathbb{N}}\|D_{r, z} D_{s, y} u_{n}(t, x) \|^2_{p} \leq {C}_{T, p}K^2_{r,z,s,y}(t,x).\label{3.9}
\end{align}
In particular, if $\sigma(x)=x$, the second part of (\ref{3.6}) vanishes. According to the preceding arguments, we have
$$\sup _{n \in \mathbb{N}}\|D_{r, z} D_{s, y} u_{n}(t, x) \|_{p} \leq {C}_{T, p}(t-s)^{-\frac{1}{2\alpha}}G^{\frac{1}{2}}_{\alpha}(t-s,x-y)
(s-r)^{-\frac{1}{2\alpha}}G^{\frac{1}{2}}_{\alpha}(s-r,y-z).$$
Moreover, using Minkowski's inequality and Lemma \ref{lem6.2} we derive that
\begin{align}
\sup _{n \in \mathbb{N}} \E\left[\|D^{2} u_{n}(t, x)\|^p_{{\mathcal H}\otimes{\mathcal H}}\right] & \leq \sup _{n \in \mathbb{N}}\left(\int_{[0, t]^{2}} \int_{\mathbb{R}^{2}}\left\|D_{r, z} D_{s, y} u_{n}(t, x)\right\|_{p}^{2} \d y \d z \d r \d s\right)^{\frac{p}{2}} \notag\\
& \leq {C}_{T, p}\left(2 \int_{0}^{t} \int_{0}^{s} \int_{\mathbb{R}^{2}} K^2_{r,z,s,y}(t,x) \d z \d y \d r \d s\right)^{\frac{p}{2}}<\infty.\label{3.10}
\end{align}
Finally, since $u_n(t,x)$ converges to $u(t,x)$ in $L^p(\Omega)$ for $p\geq2$ (see in the proof of \cite[Theorem 2.1]{Ass22}), we deduce that $u(t,x)\in\cap_{p \geq 2} \mathbb{D}^{2, p}$ by \cite[Lemma 1.5.3]{Nualart06}. Following from a similar approximation argument as the proof of Theorem 6.4 in \cite{Chen221}, we prove the results.
\end{proof}

\section{Negetive moments}
We will show estimates on the negative moments of $D_{v}F$ in this section. The following lemmas play an important role in proving Theorems \ref{thm1.1}-\ref{thm1.3}.
\begin{prop}\label{prop4.1}
{\rm ({Case} 1)}. Let $F_{N,1}$ denote the spatial average defined in (\ref{1.7}). Then, for any $p\geq2$,
\begin{align}
\sup _{N \geq 1} \mathrm{E}\left[\left|D_{v_{N,1}} F_{N,1}\right|^{-p}\right]<\infty .
\label{4.1}
\end{align}
\end{prop}
\begin{prop}\label{prop4.2}
{\rm ({Case} 2)}. Let $F_{N,2}$ denote the spatial average defined in (\ref{1.8}). Then, for any $p\geq2$,
\begin{align}
\sup _{N \geq 1} \mathrm{E}\left[\left|D_{v_{N,2}} F_{N,2}\right|^{-p}\right]<\infty .
\label{4.2}
\end{align}
\end{prop}
\begin{prop}\label{prop4.3}
{\rm ({Case} 3)}. Let $F_{N,3}$ denote the spatial average defined in (\ref{1.9}) Then, for any $p\geq2$, $\gamma>5p$,
\begin{align}
\sup _{N \geq e} \mathrm{E}\left[\left|D_{v_{N,3}} F_{N,3}\right|^{-p}\right]\leq C_{t,p,\gamma}(\log N)^\gamma.
\label{4.3}
\end{align}
\end{prop}
For the sake of simplicity, we put (\ref{1.3})-(\ref{1.5}) together and recast the mild solution in cases 1-3 as
\begin{align}
\widetilde{u}(t,x)=1+\int_{(0,t)\times{\mathbb R}^{d}}\widetilde{G}(t-s,x-y)\widetilde{u}(s,y)\eta({\rm d}s,{\rm d}y)~~~{\rm for}~t>0~{\rm and}~x\in {\mathbb R}^{d},
\label{4.4}
\end{align}
where $\widetilde{u}(t,x)$, $\widetilde{G}(t-s,x-y)$ and $\eta$ depend on the situation of cases 1-3.
Similarly, (\ref{1.7})-(\ref{1.9}) can be recast as
\begin{align}
F_{N}:=\frac{1}{\sigma_{N}}\left(\int_{[0,N]^d}\left(\u(t,x)-1\right)\d x\right),~{\rm where}~{\sigma^2_{N}}:=\Var\left(\int_{[0,N]^d} \u(t,x)\d x\right).
\label{4.5}
\end{align}
Substituting (\ref{4.4}) into (\ref{4.5}), we have
\begin{align}
F_{N} & =\frac{1}{\sigma_{N}}\left(\int_{[0,N]^d} \int_{(0,t)\times{\mathbb R}^{d}} \G(t-s,x-y)\u(s, y) \eta(\d s, \d y) \d x\right) \notag\\
& =\int_{0}^{t} \int_{\mathbb{R}^d} \frac{1}{\sigma_{N}}\left(\int_{[0,N]^d}\G(t-s,x-y)\u(s, y) \d x\right) \eta(\d s, \d y)=\delta\left(v_{N}\right)\notag,
\end{align}
where
\begin{align}
v_{N}(s, y)=\mathbf{1}_{[0, t]}(s) \frac{1}{\sigma_{N}} \int_{[0,N]^d} \G(t-s,x-y)\u(s, y)  \d x.\label{4.6}
\end{align}
Consider the Malliavin derivative of $F_{N}$,
\begin{align*}
D_{s,y}F_{N} & =\frac{1}{\sigma_{N}}\int_{[0,N]^d} D_{s,y}\u(t,x) \d x.
\end{align*}
Since $\u(t,x)\geq0$ and $D_{s,y}\u(t,x)\geq0$ for almost all $0<s<t$ and $x,y\in\R^d$ (see \cite[Theorem 3.2]{Chen23} for case 2, others can be obtained similarly), together with (\ref{4.6}), we have
\begin{align}
&D_{v_{N}} F_{N}=\int_{0}^{t} \int_{\mathbb{R}^{2d}}  v_{N}(s, y+y')D_{s, y} F_{N} \d y f(\d y') \d s \notag\\
&=\frac{1}{\sigma_{N}^{2}} \int_{0}^{t}\d s \int_{\mathbb{R}^{2d}}\d y f(\d y')\int_{[0,N]^{2d}}\d x_1 \d x_2\notag\\
&~~~~~~~~~~~~~~~~~~~~~~~~~~~~~~~~~~~~~~~~~~\times D_{s, y} \u(t, x_{1}) \G(t-s,x_2-(y+y'))\u(s, y+y')\label{4.6.5}\\
&\geq\frac{1}{\sigma_{N}^{2}} \int_{t_\a}^{t}\d s \int_{\mathbb{R}^{2d}}\d y f(\d y')\int_{[0,N]^{2d}}\d x_1 \d x_2 D_{s, y} \u(t, x_{1}) \G(t-s,x_2-(y+y'))\u(s, y+y'),\notag
\end{align}
where $t_\a:=t-\e^\a$, for any $0<\a<1$ and $0<\e^\a<\frac{t}{2}$. Recall (\ref{2.5}), (\ref{2.6}) and (\ref{2.11}), thanks to a stochastic Fubini argument, we obtain
\begin{align}
&\frac{1}{\sigma_{N}^{2}}\int_{t_\a}^{t}\d s \int_{\mathbb{R}^{2d}}\d y f(\d y')\int_{[0,N]^{2d}}\d x_1 \d x_2 D_{s, y} \u(t, x_{1}) \G(t-s,x_2-(y+y'))\u(s, y+y')\notag\\
=&I_{1}+I_{2}\notag,
\end{align}
where
\begin{align}
I_{1}:=&\frac{1}{\sigma_{N}^{2}}\int_{t_\a}^{t}\d s \int_{\mathbb{R}^{2d}}\d y f(\d y')\int_{[0,N]^{2d}}\d x_1 \d x_2\notag\\
&~~~~~~~~~~~~~~~~~~\times\G(t-s,x_1-y) \G(t-s,x_2-(y+y'))\u(s, y)\u(s, y+y'),\label{4.7}\\
I_{2}:=&\frac{1}{\sigma_{N}^{2}}\int_{[t_\a,t]\times\R^d}\eta(\d r, \d z)\int_{t_\a}^r \d s \int_{\mathbb{R}^{2d}}\d y f(\d y')\int_{[0,N]^{2d}}\d x_1 \d x_2\notag\\
&~~~~~~~~~~~~~~~~~~\times\G(t-r,x_1-z)\G(t-s,x_2-(y+y'))\u(s, y+y')D_{s,y}\u(r,z)\label{4.8}.
\end{align}
Hence, by Chebyshev's inequality, for any $q\geq2$, we have
\begin{align}
\P\left(D_{v_{N}} F_{N}<\e\right)&\leq \P\left(I_1+I_2<\e\right)
\leq \P\left({I_1}<2\e\right)+\P\left({|I_2|}>\e\right)\notag\\
&\leq\left(2\e\right)^{q}\E\left[\vert{I_1}\vert^{-q}\right]+\e^{-q}\E\left[\vert{I_2}\vert^{q}\right].\label{4.9}
\end{align}
Now we begin to prove Propositions \ref{prop4.1}-\ref{prop4.3} by estimating $\E\left[\vert{I_1}\vert^{-q}\right]$ and $\E\left[\vert{I_2}\vert^{q}\right]$.

\begin{proof}[Proof of Proposition 4.1]
In case 1, we define
\begin{align}
\phi_{N}(s,y)=\int_{[0,N]}G_\a(t-s,x-y)\d x.\notag
\end{align}
Recall the definition of $I_1$ in (\ref{4.7}), thanks to Jensen's inequality we have
\begin{align*}
\E\left[\vert{I_1}\vert^{-q}\right]
&=\E\left[\left\vert\frac{1}{\sigma^2_{N,1}}\int_{t_\a}^{t}\int_{\mathbb{R}}\phi_N^2(s,y)u^2(s, y)\d y\d s \right\vert^{-q}\right]\notag\\
&=\left(\int_{t_\a}^tM_1(s,N)\d s\right)^{-q}
\E\left[\left\vert\frac{\int_{t_\a}^{t}\int_{\mathbb{R}} \phi_N^2(s,y)u^2(s, y)\d y\d s }{{\sigma^2_{N,1}}\int_{t_\a}^tM_1(s,N)\d s}\right\vert^{-q}\right]\notag\\
&\leq \left(\int_{t_\a}^tM_1(s,N)\d s\right)^{-q-1}\frac{1}{\sigma^2_{N,1}}
\int_{t_\a}^{t}\int_{\mathbb{R}}\phi_N^2(s,y) \E\left[u^{-2q}(s, y)\right]\d y\d s,\notag
\end{align*}
where
\begin{align}
M_1(s,N):=\frac{1}{\sigma^2_{N,1}}\int_{\mathbb{R}}  \phi_N^2(s,y) \d y. \label{4.10}
\end{align}
From Lemma \ref{lem2.2} part (1) and part (3),
\begin{align}
\sup_{s\in[t/2,t]}\E\left[u^{-2q}(s, y)\right]=\sup_{s\in[t/2,t]}\E\left[u^{-2q}(s, 0)\right]\leq C_{t,q}<\infty.\notag
\end{align}
Hence, we derive that
\begin{align}
\E\left[\vert{I_1}\vert^{-q}\right]\leq C_{t,q}\left(\int_{t_\a}^tM_1(s,N)\d s\right)^{-q}.\label{4.11}
\end{align}
For every real number $N>0$, define the following functions:
\begin{align}
I_{N}(x):=N^{-d}\mathbf{1}_{{[0,N]}^{d}}(x),~~~{\widetilde{I}}_{N}(x):=I_{N}(-x)~~~{\rm for}~x\in{\mathbb R}^{d}.\label{4.12}
\end{align}
Then, we obtain that the Fourier transform of $I_{N}\ast {\widetilde{I}}_{N}$ is ${2^{-d}}\prod_{j=1}^{d}\frac{1-\cos (Nz_j)}{(Nz_j)^2}$. Note that in this setting, the functions $I_{N}\ast {\widetilde{I}}_{N}$ and $G_\a(t, \bullet)$ belong to $L^2(\R)$. According to the semigroup property of $G_\a$ and Parseval's identity, we have that for any $N\geq1$,
\begin{align}
M_1(s,N)&=\frac{1}{\sigma^2_{N,1}}\int_{[0,N]^{2}}G_\a(2(t-s),x_2-x_1)\d x_1 \d x_2\notag\\
&=\frac{N^2}{\sigma^2_{N,1}}\int_{\R}\left(I_{N}\ast {\widetilde{I}}_{N}\right)(x)G_\a(2(t-s),x)\d x\notag\\
&=\frac{N^2}{\pi\sigma^2_{N,1}}\int_{\R}\frac{1-\cos(Nz)}{(Nz)^2}e^{-{2(t-s)}z^\a}\d z\notag\\
&=\frac{N}{\pi\sigma^2_{N,1}}\int_{\R}\frac{1-\cos z}{z^2}\exp\left(-\frac{2(t-s)z^\a}{N^2}\right)\d z\label{4.13}\\
&\geq\frac{N}{\pi\sigma^2_{N,1}}\int_{[1,2]}\frac{1-\cos z}{z^2}\exp\left(-\frac{2(t-s)z^\a}{N^2}\right)\d z \geq C_t,\label{4.14}
\end{align}
where we use Lemma \ref{lem2.5} part (2) in the last inequality. Then, from (\ref{4.11}) and (\ref{4.14}), we conclude that
\begin{align}
\E\left[\vert{I_1}\vert^{-q}\right]\leq C_{t,q}\left(\int_{t_\a}^t\d s\right)^{-q}\leq C_{t,q}\e^{-\a q}.\label{4.15}
\end{align}
Now, we estimate the term $\E\left[\vert{I_2}\vert^{q}\right]$. Recall the definition of $I_2$ in (\ref{4.8}), we apply the BDG inequality and Minkowski's inequality to find that
\begin{align}
\E\left[\vert{I_2}\vert^{q}\right]&\leq C_q\mathrm{E}\left[\left| \int_{t_{\alpha}}^{t} \int_{\mathbb{R}} \left(\frac{1}{\sigma^2_{N,1}} \int_{t_{\alpha}}^{r} \int_{\mathbb{R}} \phi_{N}(s, y) \phi_{N}(r, z) u(s,y) D_{s, y} u(r, z) \d y \d s\right)^{2} \d z \d r\right|^{\frac{q}{2}}\right] \notag\\
&=C_q\mathrm{E}\left[\left|\frac{1}{\sigma^4_{N,1}}\int_{t_{\alpha}}^{t}\d r \int_{[t_{\alpha},r]^2} \d s_1 \d s_2\int_{\mathbb{R}^{3}}\d y_1 \d y_2 \d z \phi_{N}(s_{1}, y_{1}) \phi_{N}(s_{2}, y_{2}) \phi_{N}^{2}(r, z)\right.\right. \notag\\
&~~~~~~\times u(s_1,y_1)u(s_2,y_2)D_{s_{1}, y_{1}} u(r, z) D_{s_{2}, y_{2}} u(r, z)  \Bigg|^{\frac{q}{2}}\Bigg] \notag\\
&\leq C_q\left(\frac{1}{\sigma^4_{N,1}}\int_{t_{\alpha}}^{t}\d r \int_{[t_{\alpha},r]^2} \d s_1 \d s_2\int_{\mathbb{R}^{3}}\d y_1 \d y_2 \d z \phi_{N}(s_{1}, y_{1}) \phi_{N}(s_{2}, y_{2}) \phi_{N}^{2}(r, z)\right. \notag\\
&~~~~~~\times \left\|u(s_1,y_1)u(s_2,y_2)D_{s_{1}, y_{1}} u(r, z) D_{s_{2}, y_{2}} u(r, z)\right\|_{\frac{q}{2}} \Bigg)^{\frac{q}{2}}\notag\\
&\leq C_{t,q}\left(\frac{1}{\sigma^4_{N,1}}\int_{t_{\alpha}}^{t}\d r \int_{[t_{\alpha},r]^2} \d s_1 \d s_2\int_{\mathbb{R}^{3}}\d y_1 \d y_2 \d z \phi_{N}(s_{1}, y_{1}) \phi_{N}(s_{2}, y_{2}) \phi_{N}^{2}(r, z)\right. \notag\\
&~~~~~~\times (r-s_1)^{-\frac{1}{2\a}}G_\a^{\frac{1}{2}}(r-s_1,z-y_1)
(r-s_2)^{-\frac{1}{2\a}}G_\a^{\frac{1}{2}}(r-s_2,z-y_2)\Bigg)^{\frac{q}{2}},\label{4.16}
\end{align}
where we use Lemma \ref{lem2.2} part (1) and part (2) in the last inequality. Notice that $\phi_N(s,y)\leq1$ by Lemma \ref{lem6.1} part (3). Thanks to Lemma \ref{lem6.1} part (4), we obtain that for any $N\geq1$,
\begin{align}
&\frac{1}{\sigma^4_{N,1}}\int_{\mathbb{R}^{3}}\phi_{N}^{2}(r, z)\prod_{i=1,2}\phi_{N}(s_{i}, y_{i})
(r-s_i)^{-\frac{1}{2\a}}G_\a^{\frac{1}{2}}(r-s_i,z-y_i)\d y_1 \d y_2 \d z\notag\\
&\leq\frac{1}{\sigma^4_{N,1}}\int_{\mathbb{R}^{3}}\phi_{N}^{2}(r, z)\prod_{i=1,2}(r-s_i)^{-\frac{1}{2\a}}G_\a^{\frac{1}{2}}(r-s_i,z-y_i)\d y_1 \d y_2 \d z\notag\\
&\leq \frac{C}{\sigma^4_{N,1}}\int_{\mathbb{R}}\phi_{N}^{2}(r, z)\d z= \frac{CN}{\pi\sigma^4_{N,1}}\int_{\R}\frac{1-\cos z}{z^2}\exp\left(-\frac{2(t-r)z^\a}{N^2}\right)\d z\leq C_t,\label{4.17}
\end{align}
where the equality holds by (\ref{4.13}), and we use $\int_{\mathbb{R}}\frac{1-\cos z}{z^2}\d z=\pi$ and Lemma \ref{lem2.5} part (2) in the last inequality. Then, substituting (\ref{4.17}) into (\ref{4.16}), we conclude that
\begin{align}
\E\left[\vert{I_2}\vert^{q}\right]\leq C_{t,q}\left(\int_{t_{\alpha}}^{t}\d r \int_{[t_{\alpha},r]^2} \d s_1 \d s_2 \right)^{\frac{q}{2}}\leq C_{t,q}\e^{\frac{3\a q}{2}}.\label{4.18}
\end{align}
Choose $\a=4/5$. (\ref{4.9}), (\ref{4.15}) and (\ref{4.18}) together imply that
\begin{align}
\sup_{N\geq 1}\P\left(D_{v_{N,1}} F_{N,1}<\e\right)\leq C_{t,q}\e^{\frac{q}{5}}.\label{4.19}
\end{align}
Therefore, we finally get
\begin{align}
\sup_{N\geq 1}\E\left[\left(D_{v_{N,1}} F_{N,1}\right)^{-p}\right]
&=\sup_{N\geq1}\int_0^\infty p\e^{-p-1}\P\left(D_{v_{N,1}} F_{N,1}<\e\right)\d \e\notag\\
&\leq1+C_{t,q}p\int_0^1\e^{-p-1+q/5}\d \e<\infty,\label{4.20}
\end{align}
for all $q>5p$. This proves our result.
\end{proof}

\begin{proof}[Proof of Proposition 4.2]
In case 2, recall the definition of $I_1$ in (\ref{4.7}), thanks to Jensen's inequality, we have
\begin{align}
\E\left[\vert{I_1}\vert^{-q}\right]
&=\E\left[\left\vert\frac{1}{\sigma^2_{N,2}}
\int_{t_\a}^{t}\d s\int_{\mathbb{R}^{2d}}\d y f(\d y')\int_{[0,N]^{2d}}\d x_1 \d x_2
\right.\right.\notag\\
&~~~~~~\times p_{t-s}(x_1-y)p_{t-s}(x_2-(y+y'))u(s,y)u(s,y+y') \Bigg|^{-q}\Bigg] \notag\\
&\leq\left(\int_{t_\a}^tM_2(s,N)\d s\right)^{-q-1}
\int_{t_\a}^{t}\d s\int_{\mathbb{R}^{2d}}\d yf(\d y')\int_{[0,N]^{2d}}\d x_1 \d x_2
\notag\\
&~~~~~~\times \frac{1}{\sigma^2_{N,2}}p_{t-s}(x_1-y)p_{t-s}(x_2-(y+y')) \E\left[u^{-q}(s, y)u^{-q}(s,y+y')\right],  \notag
\end{align}
where
\begin{align}
M_2(s,N):=\frac{1}{\sigma^2_{N,2}}\int_{\mathbb{R}^{2d}}\d y f(\d y')\int_{[0,N]^{2d}}\d x_1 \d x_2p_{t-s}(x_1-y)p_{t-s}(x_2-(y+y')).\label{4.21}
\end{align}
Thanks to H$\rm \ddot{o}$lder's inequality, Lemma \ref{lem2.3} part (1) and part (4), we can see that
\begin{align}
\sup_{s\in[t/2,t]}\E\left[u^{-q}(s, y)u^{-q}(s,y+y')\right]&\leq\sup_{s\in[t/2,t]}\left[\E\left(u^{-2q}(s, y)\right)\E\left(u^{-2q}(s, y+y')\right)\right]^{\frac{1}{2}}\notag\\
&\leq \sup_{s\in[t/2,t]}\E\left[u^{-2q}(s, 0)\right]\leq C_{t,q}<\infty.\label{4.22}
\end{align}
This, together with Lemma \ref{lem6.3} part (1), concludes that
\begin{align}
\E\left[\vert{I_1}\vert^{-q}\right]\leq C_{t,q}\left(\int_{t_\a}^tM_2(s,N)\d s\right)^{-q}
\leq C_{t,q}\left(\int_{t_\a}^t\d s\right)^{-q}
\leq C_{t,q}\e^{-\a q},\label{4.23}
\end{align}
for any $N\geq1$. As for $\E\left[\vert{I_2}\vert^{q}\right]$, recall (\ref{4.8}), using the BDG inequality and Minkowski's inequality, we can write
\begin{align}
\E\left[\vert{I_2}\vert^{q}\right]
&\leq C_{q}\left(\int_{t_{\alpha}}^{t}\d r \int_{[t_{\alpha},r]^2} \d s_1 \d s_2\int_{\mathbb{R}^{6d}}\d y_1 f(\d y_1')\d y_2 f(\d y_2')\d z f(\d z')\int_{[0,N]^{4d}}\d x_1 \d x_1'\d x_2 \d x_2'\right. \notag\\
&~~~~~~\times \frac{1}{\sigma^4_{N,2}} p_{t-r}(x_1-z)p_{t-r}(x_1'-(z+z'))p_{t-s_1}(x_2-(y_1+y_1'))p_{t-s_2}(x_2'-(y_2+y_2'))\notag\\
&~~~~~~\times \|u(s_1,y_1+y_1')u(s_2,y_2+y_2')D_{s_{1}, y_{1}} u(r, z) D_{s_{2}, y_{2}} u(r, z+z')\|_{\frac{q}{2}}\Bigg)^{\frac{q}{2}} \notag\\
&\leq C_{t,q}\left(\int_{t_{\alpha}}^{t}\d r \int_{[t_{\alpha},r]^2} \d s_1 \d s_2\int_{\mathbb{R}^{6d}}\d y_1 f(\d y_1')\d y_2 f(\d y_2')\d z f(\d z')\int_{[0,N]^{4d}}\d x_1 \d x_1'\d x_2 \d x_2'\right. \notag\\
&~~~~~~\times \frac{1}{\sigma^4_{N,2}} p_{t-r}(x_1-z)p_{t-r}(x_1'-(z+z'))p_{t-s_1}(x_2-(y_1+y_1'))p_{t-s_2}(x_2'-(y_2+y_2'))\notag\\
&~~~~~~\times p_{r-s_1}(z-y_1)p_{r-s_2}(z+z'-y_2)\Bigg)^{\frac{q}{2}},\label{4.24}
\end{align}
where we use Lemma \ref{lem2.3} part (2) in the last inequality. Then, we proceed in the following order:
integrating in $y_1$, $y_2$ and using the semigroup property of the heat kernel; integrating $x_2$, $x_2'$ on $\R^d$, to obtain that for any $N\geq1$,
\begin{align}
\E\left[\vert{I_2}\vert^{q}\right]&\leq C_{t,q}
\left(\frac{1}{\sigma^4_{N,2}} \int_{t_{\alpha}}^{t}\d r \int_{[t_{\alpha},r]^2} \d s_1 \d s_2
\int_{\mathbb{R}^{4d}}f(\d y_1')f(\d y_2')\d z f(\d z')\int_{[0,N]^{2d}}\d x_1 \d x_1'\right.\notag\\
&~~~~~~\times p_{t-r}(x_1-z)p_{t-r}(x_1'-(z+z'))\Bigg)^{\frac{q}{2}}\notag\\
&\leq C_{t,q}\left(\frac{\e^{2\a}(f(\R^d))^2}{\sigma^2_{N,2}} \int_{t_{\alpha}}^{t}M_2(r,N)\d r\right)^{\frac{q}{2}}\leq C_{t,q} \e^{\frac{3\a q}{2}},\label{4.25}
\end{align}
where we use Lemma \ref{lem6.3} part (2) and Lemma \ref{lem2.5} part (3) in the last inequality. Choose $\a=4/5$. Similar to the argument in (\ref{4.20}), from (\ref{4.9}), (\ref{4.23}) and (\ref{4.25}), we finally prove the result.
\end{proof}

\begin{proof}[Proof of Proposition 4.3]
In case 3, we first estimate the term $\E\left[\vert{I_1}\vert^{-q}\right]$. Similar to the proof of Proposition 4.2, we apply Jensen's inequality, Lemma \ref{lem2.4} part (1) and part (4) to see that
\begin{align}
\E\left[\vert{I_1}\vert^{-q}\right]\leq C_{t,q}\left(\int_{t_\a}^tM_3(s,N)\d s\right)^{-q},
\end{align}
where
\begin{align}
M_3(s,N):=\frac{1}{\sigma^2_{N,3}}\int_{\mathbb{R}^2}\d y f(\d y')\int_{[0,N]^{2}}\d x_1 \d x_2p_{\frac{s(t-s)}{t}}(y-\frac{s}{t}x_1)p_{\frac{s(t-s)}{t}}(y+y'-\frac{s}{t}x_2).
\label{4.27}
\end{align}
From Lemma \ref{lem6.4} part (1), we conclude that
\begin{align}
\E\left[\vert{I_1}\vert^{-q}\right]
\leq C_{t,q}(\log N)^{q}\left(\int_{t_\a}^t\d s\right)^{-q}
\leq C_{t,q}\e^{-\a q}(\log N)^q,\label{4.28}
\end{align}
for any $N\geq e$. Next, recall (\ref{4.8}), Lemma \ref{lem2.4} part (1) and part (3). By the BDG inequality and Minkowski's inequality,
\begin{align}
\E[\vert{I_2}&\vert^{q}]\leq C_{t,q}\left(\int_{t_{\alpha}}^{t}\d r \int_{[t_{\alpha},r]^2} \d s_1 \d s_2\int_{\mathbb{R}^{6}}\d y_1 f(\d y_1')\d y_2 f(\d y_2')\d z f(\d z')\int_{[0,N]^{4}}\d x_1 \d x_1'\d x_2 \d x_2'\right. \notag\\
\times& \frac{1}{\sigma^4_{N,3}} p_{\frac{r(t-r)}{t}}\left(z-\frac{r}{t}x_1\right)
p_{\frac{r(t-r)}{t}}\left(z+z'-\frac{r}{t}x_1'\right)
p_{\frac{s_1(t-s_1)}{t}}\left(y_1+y_1'-\frac{s_1}{t}x_2\right)
\notag\\
\times& p_{\frac{s_2(t-s_2)}{t}}\left(y_2+y_2'-\frac{s_2}{t}x_2'\right)
p_{\frac{s_1(r-s_1)}{r}}\left(y_1-\frac{s_1}{r}z\right)
p_{\frac{s_2(r-s_2)}{r}}\left(y_2-\frac{s_2}{r}(z+z')\right)
\Bigg)^{\frac{q}{2}}.\label{4.29}
\end{align}
Then, we use the same arguments in proving (\ref{4.25}) and apply the following identity in integrating $x_2$, $x_2'$:
\begin{align}
p_t(\theta x)=\theta^{-d}p_{t/\theta^2}(x), {\rm~for~all~}x\in\R~{\rm and}~t,\sigma>0.\label{4.30}
\end{align}
As a consequence, for any $N\geq e$,
\begin{align}
\E\left[\vert{I_2}\vert^{q}\right]
\leq C_{t,q}\left(\frac{(f(\R))^2}{\sigma^2_{N,3}} \int_{t_{\alpha}}^{t}M_3(r,N)\d r\int_{[t_{\alpha},r]^2} \frac{1}{s_1s_2}\d s_1 \d s_2\right)^{\frac{q}{2}}\leq C_{t,q} \e^{\frac{3\a q}{2}},\label{4.31}
\end{align}
where the last inequality holds by $t_\a>t/2$ and Lemma \ref{lem6.4} part (2). Choose $\a=4/5$. Combining (\ref{4.9}), (\ref{4.28}) and (\ref{4.31}), we finally get
\begin{align}
\E\left[\left(D_{v_{N,3}} F_{N,3}\right)^{-p}\right]
\leq1+C_{t,q}\left(\log N\right)^{q}p\int_0^1\e^{-p-1+q/5}\d \e\leq C_{t,q}\left(\log N\right)^{q},\label{4.32}
\end{align}
for all $q>5p$ and $N\geq e$. This proves the result.
\end{proof}
\section{Proofs of Theorems 1.1-1.3}
In this section, we will establish upper bounds for the uniform distance of densities and prove Theorems \ref{thm1.1}-\ref{thm1.3} by analyzing the behavior of $\|D_{v_{N}}(D_{v_{N}}F_{N})\|_2$.

In cases 1-3, recall (\ref{4.6.5}) that
\begin{align}
D_{v_{N}} F_{N}=\frac{1}{\sigma_{N}^{2}} &\int_{0}^{t}\d s \int_{\mathbb{R}^{2d}}\d y f(\d y')\int_{[0,N]^{2d}}\d x_1 \d x_2 \notag\\
&~~~~~~\times D_{s, y} \u(t, x_{1}) \G(t-s,x_2-(y+y'))\u(s, y+y').\label{5.1}
\end{align}
Applying the Malliavin derivative operator, we have
\begin{align}
D_{r, z}(D_{v_{N}}F_{N})&= \frac{1}{\sigma_{N}^{2}} \int_{0}^{t}\d s\int_{\mathbb{R}^{2d}}\d y f(\d y')\int_{[0,N]^{2d}}\d x_1 \d x_2\G(t-s,x_2-(y+y'))\notag\\
&~~~~~~\times \left(D_{s, y}\u(t, x_{1})
 D_{r, z} \u(s,y+y')
+\u(s,y+y')D_{r,z}D_{s,y}\u(t,x_1)\right).\notag
\end{align}
Recall (\ref{4.6}), we obtain
\begin{align}
D_{v_{N}}(D_{v_{N}} &F_{N})= \frac{1}{\sigma_{N}^{3}} \int_{0}^{t}\d r \int_r^t\d s\int_{\mathbb{R}^{4d}}\d z f(\d z')\d y f(\d y')\int_{[0,N]^{3d}}\d x_1 \d x_2 \d x_3\notag\\
&\times \G(t-s,x_2-(y+y'))\G(t-r,x_3-(z+z'))\u(r,z+z')\notag\\
&\times \left(D_{s, y}\u(t, x_{1})
 D_{r, z} \u(s,y+y')
+2\u(s,y+y')D_{r,z}D_{s,y}\u(t,x_1)\right).\label{5.2}
\end{align}
Now we begin to prove Theorems \ref{thm1.1}-\ref{thm1.3}.

\begin{proof}[Proof of Theorem 1.1.]
In case 1, according to the proof of \cite[Theorem 2.3]{Ass22}, it is easy to see that
\begin{align}
\left\|1-D_{v_{N,1}} F_{N,1}\right\|_{2} \leq \frac{C_{t}}{\sqrt{N}}.\label{5.3}
\end{align}
Recall Lemma \ref{lem2.1} and Lemma \ref{lem2.5} part (1), it remains to estimate the term $\|D_{v_{N,1}}(D_{v_{N,1}}F_{N,1})\|_2$. According to H${\rm\ddot{o}}$lder's inequality, Proposition \ref{prop3.1}, Lemma \ref{lem2.2} part (1) and part (2), we have
\begin{align}
&\left\|u(r,z)\left(D_{s, y}u(t, x_{1})
 D_{r, z}u(s,y)
+2u(s,y)D_{r,z}D_{s,y}u(t,x_1)\right)\right\|_2\notag\\
&~~~~~~\leq C_t
(t-s)^{-\frac{1}{2\alpha}}G^{\frac{1}{2}}_{\alpha}(t-s,x_1-y)
(s-r)^{-\frac{1}{2\alpha}}G^{\frac{1}{2}}_{\alpha}(s-r,y-z).\label{5.4}
\end{align}
Hence, recall (\ref{5.2}), thanks to Minkowski's inequality,
\begin{align}
&\|D_{v_{N,1}}(D_{v_{N,1}}F_{N,1})\|_2\leq\frac{C_t}{\sigma_{N,1}^{3}} \int_{0}^{t}\d r \int_r^t\d s\int_{\mathbb{R}^{2}}\d z\d y \int_{[0,N]^{3}}\d x_1 \d x_2 \d x_3\notag\\
&~~~\times G_\a(t-s,x_2-y)G_\a(t-r,x_3-z)(t-s)^{-\frac{1}{2\alpha}}G^{\frac{1}{2}}_{\alpha}(t-s,x_1-y)
(s-r)^{-\frac{1}{2\alpha}}G^{\frac{1}{2}}_{\alpha}(s-r,y-z).\notag
\end{align}
Then, we proceed in the following order: integrating $x_3$ on $\R$ by using Lemma \ref{lem6.1} part (3); integrating $x_1$ and $z$ on $\R$ by using Lemma \ref{lem6.1} part (4); integrating $y$ on $\R$, to obtain that for any $N\geq 1$,
\begin{align}
\|D_{v_{N,1}}(D_{v_{N,1}}F_{N,1})\|_2\leq\frac{C_t}{\sigma_{N,1}^{3}} \int_{0}^{t}\d r \int_r^t\d s \int_{[0,N]}\d x_2
\leq\frac{C_t}{\sqrt{N}},
\end{align}
where the last inequality holds by Lemma \ref{lem2.5} part (2). This proves Theorem \ref{thm1.1}.
\end{proof}

\begin{proof}[Proof of Theorem 1.2.]
In case 2, from the proof of \cite[Theorem 2.5]{Chen23}, we have
\begin{align}
\left\|1-D_{v_{N,2}} F_{N,2}\right\|_{2} \leq \frac{C_{t}}{(\sqrt{N})^{d}}.\label{5.6}
\end{align}
Thanks to Lemma \ref{lem2.1} and Lemma \ref{lem2.5} part (1), we estimate the term $\|D_{v_{N,2}}(D_{v_{N,2}}F_{N,2})\|_2$ in the following. According to H${\rm\ddot{o}}$lder's inequality, Lemma \ref{lem2.3} part (2) and part (3), we have
\begin{align}
&\left\|u(r,z+z')\left(D_{s, y}u(t, x_{1})
 D_{r, z}u(s,y+y')
+2u(s,y+y')D_{r,z}D_{s,y}u(t,x_1)\right)\right\|_2\notag\\
&~~~~~~\leq C_t p_{t-s}(x_1-y)p_{s-r}(y+y'-z)+C_t p_{t-s}(x_1-y)p_{s-r}(y-z).\label{5.7}
\end{align}
Then, recall (\ref{5.2}), we have
\begin{align}
\|D_{v_{N,2}}(D_{v_{N,2}}F_{N,2})\|_2\leq\Phi_{N,1}+\Phi_{N,2},\label{5.8}
\end{align}
where
\begin{align}
\Phi_{N,1}&=\frac{C_t}{\sigma_{N,2}^{3}} \int_{0}^{t}\d r \int_r^t\d s\int_{\mathbb{R}^{4d}}\d z f(\d z')\d y f(\d y')\int_{[0,N]^{3d}}\d x_1 \d x_2 \d x_3\notag\\
&~~~~~~\times p_{t-s}(x_2-(y+y'))p_{t-r}(x_3-(z+z'))p_{t-s}(x_1-y)p_{s-r}(y+y'-z),\notag\\
\Phi_{N,2}&=\frac{C_t}{\sigma_{N,2}^{3}} \int_{0}^{t}\d r \int_r^t\d s\int_{\mathbb{R}^{4d}}\d z f(\d z')\d y f(\d y')\int_{[0,N]^{3d}}\d x_1 \d x_2 \d x_3\notag\\
&~~~~~~\times p_{t-s}(x_2-(y+y'))p_{t-r}(x_3-(z+z'))p_{t-s}(x_1-y)p_{s-r}(y-z).\notag
\end{align}
Next, for both $\Phi_{N,1}$ and $\Phi_{N,2}$, we first integrate in $z$ and use the semigroup property of the heat kernel, then integrate $x_3$ on $\R^d$ to obtain that for $i=1,2$,
\begin{align}
\Phi_{N,i}&\leq
\frac{C_t}{\sigma_{N,2}^{3}} \int_{0}^{t}\d r \int_r^t\d s\int_{\mathbb{R}^{3d}}f(\d z')\d y f(\d y')\int_{[0,N]^{2d}}\d x_1 \d x_2p_{t-s}(x_2-(y+y'))p_{t-s}(x_1-y)\notag\\
&=\frac{C_tf(\R^d)}{\sigma_{N,2}} \int_{0}^{t}\d r \int_r^tM_2(s,N)\d s\leq\frac{C_t}{(\sqrt{N})^{d}},\label{5.9}
\end{align}
where $M_2(s,N)$ is defined in (\ref{4.21}), and we use Lemma \ref{lem6.3} part (2) and Lemma \ref{lem2.5} part (3) in the last inequality. Hence, recall (\ref{5.8}), we finish the proof.
\end{proof}

\begin{proof}[Proof of Theorem 1.3.]In case 3, it follows from the proof of \cite[Theorem 1.3]{Davar21} that
\begin{align}
\left\|1-D_{v_{N,3}} F_{N,3}\right\|_{2} \leq \frac{C_{t}\sqrt{\log N}}{\sqrt{N}}.\label{5.10}
\end{align}
Since Lemma \ref{lem2.4} part (1) and part (3) hold, we repeat the computation in the proof of Theorem \ref{thm1.2} to find that
\begin{align}
\|D_{v_{N,3}}(D_{v_{N,3}}F_{N,3})\|_2\leq\Psi_{N,1}+\Psi_{N,2},\label{5.11}
\end{align}
where
\begin{align}
\Psi_{N,1}&=\frac{C_t}{\sigma_{N,3}^{3}} \int_{0}^{t}\d r \int_r^t \d s\int_{\mathbb{R}^{4}}\d z f(\d z')\d y f(\d y')\int_{[0,N]^{3}}\d x_1 \d x_2 \d x_3 p_{\frac{s(t-s)}{t}}\left(y+y'-\frac{s}{t}x_2\right)\notag\\
&~~~~~~\times p_{\frac{r(t-r)}{t}}
\left(z+z'-\frac{r}{t}x_3\right)p_{\frac{s(t-s)}{t}}\left(y-\frac{s}{t}x_1\right)
p_{\frac{r(s-r)}{s}}\left(z-\frac{r}{s}(y+y')\right),\notag\\
\Psi_{N,2}&=\frac{C_t}{\sigma_{N,3}^{3}} \int_{0}^{t}\d r \int_r^t\d s\int_{\mathbb{R}^{4}}\d z f(\d z')\d y f(\d y')\int_{[0,N]^{3}}\d x_1 \d x_2 \d x_3p_{\frac{s(t-s)}{t}}\left(y+y'-\frac{s}{t}x_2\right)\notag\\
&~~~~~~\times p_{\frac{r(t-r)}{t}}
\left(z+z'-\frac{r}{t}x_3\right)p_{\frac{s(t-s)}{t}}\left(y-\frac{s}{t}x_1\right)
p_{\frac{r(s-r)}{s}}\left(z-\frac{r}{s}y\right).\notag
\end{align}
Since $1/s$ is not integrable on $(0,t)$, we can not apply the same argument as Theorem \ref{thm1.2} in the following estimate. Hence, we first integrate $x_1$ and $x_2$ on $\R$ for $\Psi_{N,1}$ and $\Psi_{N,2}$, respectively, by using (\ref{4.30}). Then for both $\Psi_{N,1}$ and $\Psi_{N,2}$, owing to the semigroup property and (\ref{4.30}), we integrate in the variables $y$ and $z$ to obtain that for $i=1,2$,
\begin{align}
\Psi_{N,i}&\leq\frac{C_t}{\sigma_{N,3}^{3}} \int_{0}^{t}\d r \int_r^t\frac{1}{s}\d s\int_{\mathbb{R}^{2}}f(\d z')f(\d y')\int_{[0,N]^{2}}\d x_1 \d x_2\notag\\
&~~~~~~\times
p_{[({r(s-r)}/{s})(s^2/r^2)+s(t-s)/t](r^2/s^2)+r(t-r)/t}
\left(z'-\frac{r}{t}(x_2-x_1)\right)\notag\\
&=\frac{C_tf(\R)}{\sigma_{N,3}^{3}} \int_{0}^{t}\d r \int_r^t\frac{1}{s}\d s\int_{\mathbb{R}}f(\d z')\int_{[0,N]^{2}}\d x_1 \d x_2p_{\frac{2r(t-r)}{t}}
\left(z'-\frac{r}{t}(x_2-x_1)\right)\notag\\
&=\frac{C_tN^2f(\R)}{\sigma_{N,3}^{3}} \int_{0}^{t}\d r \int_r^t\frac{1}{s}\d s
\int_{\R} \d x \left(I_{N}\ast {\widetilde{I}}_{N}\right)(x)\left(p_{2r(t-r)/t}\ast f\right)\left(\frac{r}{t}x\right)\notag\\
&\leq\frac{C_tN(f(\R))^2}{\sigma_{N,3}^{3}} \int_{0}^{t}\frac{1}{r}\d r \int_r^t\frac{1}{s}\d s\int_{\R} \frac{1-\cos z}{z^2}e^{-\frac{t(t-r)z^2}{rN^2}}\d z,\label{5.12}
\end{align}
where the last inequality holds by the proof of Lemma \ref{lem6.4}. Then, integrating in the variable $s$ and making a change of variable $\sigma=(t-r)/r$ yields
\begin{align}
\Psi_{N,i}&\leq\frac{C_tN}{\sigma_{N,3}^{3}} \int_{\R}\frac{1-\cos z}{z^2}\int_{0}^{\infty}
\frac{\log (1+\sigma) }{1+\sigma}e^{-\frac{t\sigma z^2}{N^2}}\d \sigma\d z\notag\\
&=\frac{C_tN}{2\sigma_{N,3}^{3}} \int_{\R}\frac{1-\cos z}{z^2}\left(\int_{0}^{\infty}
e^{-\frac{t\sigma z^2}{N^2}}\d (\log(1+\sigma))^2\right)\d z\notag\\
&=\frac{C_tN}{2\sigma_{N,3}^{3}} \int_{\R}\frac{1-\cos z}{z^2}\left(\int_{0}^{\infty}
\frac{tz^2}{N^2}e^{-\frac{t\sigma z^2}{N^2}}(\log(1+\sigma))^2\d \sigma\right)\d z,\notag
\end{align}
where we integrate by parts in the second equality. Next, we make a change of variable $\theta=tz^2\sigma/N^2$ to obtain that
\begin{align}
\Psi_{N,i}&\leq\frac{C_tN}{\sigma_{N,3}^{3}} \int_{\R}\frac{1-\cos z}{z^2}\left(\int_{0}^{\infty}
e^{-\theta}\left(\log\left(1+\frac{N^2\theta}{tz^2}\right)\right)^2\d \theta\right)\d z\notag\\
&\leq\frac{C_tN(\log N)^2}{\sigma_{N,3}^{3}}\leq\frac{C_t\sqrt{\log N}}{\sqrt{N}},\label{5.13}
\end{align}
where we use Lemma \ref{lem6.5} in the second inequality, and the last inequality holds by Lemma \ref{lem2.5} part (4). Finally, recall Lemma \ref{lem2.1}, we combine (\ref{5.10}), (\ref{5.11}), Proposition \ref{prop4.3} and Lemma \ref{lem2.5} part (1) to finish the proof.
\end{proof}
\begin{rem}
Unlike the proof of nonlinear case \cite[Proof of Theorem 1.1]{Ku22}, when dealing with (\ref{5.2}), we apply moment inequalities of Malliavin derivative of the solutions directly rather than using the expansion of the Malliavin derivative such as (\ref{2.5}) and estimating the stochastic integral. The similar technique can be used to prove Theorem 1.2 in \cite{Ku22} to simplify the computations.
\end{rem}
\section{Appendix}

\renewcommand{\thethm}{A.\arabic{lem}}
\newcounter{lem}
We first introduce some properties of the Green kernel $G_{\a}(t,x)$ which could be found in \cite{Ass22}.
\setcounter{lem}{1}
\begin{lem}\label{lem6.1}
Let $G_{\a}(t,x)$ denote the Green kernel defined in $(\ref{1.3})$. Then,\\
${\rm (1)}$ $G^2_{\a}(t,x)\leq Ct^{-\frac{1}{\a}}G_{\a}(t,x)$, for all $(t,x)\in \R^+\times\R$.\\
${\rm (2)}$ {\rm(Semigroup property.)} $\int_{\mathbb{R}} G_{\alpha}(t, z) G_{\alpha}(s, x-z) \d z=G_{\alpha}(t+s, x)$, for $t,s>0$ and $x \in \mathbb{R}$.\\
${\rm (3)}$ $\int_{\mathbb{R}} G_{\alpha}(t, x) \d x=1$, for every $t>0$.\\
${\rm (4)}$ $\int_{\mathbb{R}} G_{\alpha}^{\frac{1}{2}}(t, x) \d x=C t^{\frac{1}{2\a}}$, for every $t>0$.
\end{lem}
\setcounter{lem}{2}
\begin{lem}\label{lem6.2}
Let $K^2_{r,z,s,y}(t,x)$ be defined in (\ref{3.2}), then for any fixed $0<r<s<t$, we have
\begin{align*}
\int_{0}^{t} \int_{0}^{s} \int_{\mathbb{R}^{2}} K^2_{r,z,s,y}(t,x) \d z \d y \d r \d s\leq C_t
\end{align*}
\end{lem}
\begin{proof}
\begin{align*}
&\int_{0}^{t} \int_{0}^{s} \int_{\mathbb{R}^{2}} K^2_{r,z,s,y}(t,x) \d z \d y \d r \d s\\
\leq& \int_{0}^{t} \int_{0}^{s}\int_{\mathbb{R}^{2}}(t-s)^{-\frac{1}{\alpha}}G_{\alpha}(t-s,x-y)(s-r)^{-\frac{1}{\alpha}}G_{\alpha}(s-r,y-z)
\d z \d y \d r \d s\\
&+\int_{0}^{t} \int_{0}^{s}\int_s^t\int_{\mathbb{R}^{3}}\left[(t-\theta)(\theta-r)(\theta-s)\right]^{-\frac{1}{\alpha}}\\
&~~~~~~\times G_{\alpha}(t-\theta,x-w)
G_{\alpha}(\theta-r,w-z)
G_{\alpha}(\theta-s,w-y) \d z \d y \d w \d \theta \d r\d s\\
=&\int_{0}^{t} \int_{0}^{s} \left[(t-s)(s-r)\right]^{-\frac{1}{\alpha}}\d r\d s+\int_{0}^{t} \int_{0}^{s}\int_s^t\left[(t-\theta)(\theta-r)(\theta-s)\right]^{-\frac{1}{\alpha}}\d \theta \d r\d s\\
\leq& C_t<\infty,
\end{align*}
where we integrate $y$, $z$, $w$ in order and use Lemma \ref{lem6.1} part (2) and part(3) in the equality.
\end{proof}
\setcounter{lem}{3}
\begin{lem}\label{lem6.3}
Let $M_2{(s,N)}$ be defined in (\ref{4.21}). Then, \\
${\rm (1)}$  $M_2{(s,N)}\geq C_t$, for all $t/2<s<t$, $N\geq 1$.\\
${\rm (2)}$  $M_2{(s,N)}\leq C_t$, for all $0<s<t$, $N\geq 1$.
\end{lem}
\begin{proof} Recall the definition of $I_{N},~{\widetilde{I}}_{N}$ in (\ref{4.12}),
\begin{align*}
M_2(s,N)&=\frac{1}{\sigma^2_{N,2}}\int_{\mathbb{R}^2}\d y f(\d y')\int_{[0,N]^{2d}}\d x_1 \d x_2p_{t-s}(x_1-y)p_{t-s}(x_2-(y+y'))\\
&=\frac{N^{2d}}{\pi^d\sigma^2_{N,2}}\int_{\R^d} \d x \left(I_{N}\ast {\widetilde{I}}_{N}\right)(x)\left(p_{2(t-s)}\ast f\right)(x)\\
&=\frac{N^{2d}}{\pi^d\sigma^2_{N,2}}\int_{\R^d} \prod_{j=1}^{d}\frac{1-\cos (Nz_j)}{(Nz_j)^2}e^{-(t-s)\|z\|^2}\hat{f}(\d z)\\
&=\frac{N^{d}}{\pi^d\sigma^2_{N,2}}\int_{\R^d} \prod_{j=1}^{d}\frac{1-\cos z_j}{z_j^2}e^{-\frac{(t-s)\|z\|^2}{N^2}}\hat{f}(\d z).
\end{align*}
(1). $f(\R^d)<\infty$ implies that $\hat{f}$ is a bounded and continuous function, then for $t/2<s<t$, choose $0<a<b<1$ such that $\inf_{z\in[a,b]}\hat{f}(z)>0$, we have
\begin{align*}
M_2(s,N)\geq\frac{N^{d}}{\pi^d\sigma^2_{N,2}}\int_{[a,b]^d} \prod_{j=1}^{d}\frac{1-\cos z}{z^2}e^{-\frac{(t-s)\|z\|^2}{N^2}}\hat{f}(z)\d z\geq C_t,
\end{align*}
where we use Lemma \ref{lem2.5} part (3) in the last inequality.\\
(2). Notice that $\hat{f}(x)\leq\hat{f}(0)=f(\R^d)$, then for all $0<s<t$,
$$
M_2(s,N)\leq\frac{N^{d}\pi^df(\R^d)}{\pi^d\sigma^2_{N,2}}\leq C_t.
$$
\end{proof}
\setcounter{lem}{4}
\begin{lem}\label{lem6.4}
Let $M_3{(s,N)}$ be defined in (\ref{4.27}). Then, \\
${\rm (1)}$  $M_3{(s,N)}\geq \frac{C_t}{\log N}$, for all $t/2<s<t$, $N\geq e$.\\
${\rm (2)}$  $M_3{(s,N)}\leq \frac{C_t}{s\log N}$, for all $0<s<t$, $N\geq e$.
\end{lem}
\begin{proof}
According to the proof of Lemma {\ref{lem6.3}}, we have
\begin{align*}
M_3(s,N)
&=\frac{N^{2}}{\sigma^2_{N,3}}\int_{\R} \d x \left(I_{N}\ast {\widetilde{I}}_{N}\right)(x)\left(p_{2s(t-s)/t}\ast f\right)\left(\frac{s}{t}x\right)\\
&=\frac{N^{2}}{\pi\sigma^2_{N,3}}\int_{\R} \frac{1-\cos (Nzs/t)}{(Nzs/t)^2}e^{-s(t-s)\|z\|^2}\hat{f}(\d z)\\
&=\frac{Nt}{s\pi\sigma^2_{N,2}}\int_{\R} \frac{1-\cos z}{z^2}e^{-\frac{t(t-s)\|z\|^2}{sN^2}}\hat{f}\left(\frac{tz}{s}\right)\d z.
\end{align*}
Then thanks to Lemma {\ref{lem2.5}} part (4), by a similar argument in proving Lemma {\ref{lem6.3}}, we prove the result.
\end{proof}
\setcounter{lem}{5}
\begin{lem}\label{lem6.5}$$\int_{\R}\frac{1-\cos z}{z^2}\left(\int_{0}^{\infty}e^{-\theta}
\left(\log\left(1+\frac{N^2\theta}{tz^2}\right)\right)^2\d \theta\right)\d z\leq C_t(\log N)^2.$$
\end{lem}
\begin{proof}
Since
\begin{align}
1+\frac{N^2\theta}{tz^2}\leq N^2+N^2\theta\cdot\frac{1}{t}\cdot\frac{1}{z^2}\leq N^2(\theta+1)\left(\frac{1}{t}+1\right)\left(\frac{1}{z^2}+1\right),\notag
\end{align}
then, we have
$$\log\left(1+\frac{N^2\theta}{tz^2}\right)\leq\left(2\log N+\log \left(\frac{1}{t}+1\right)\right)\cdot\left(1+\log(\theta+1)
+\log\left(\frac{1}{z^2}+1\right)\right).$$
Notice that
\begin{align}
\int_{\R}\frac{1-\cos z}{z^2}\left(\int_{0}^{\infty}e^{-\theta}
\left(1+\log(\theta+1)
+\log\left(\frac{1}{z^2}+1\right)\right)^2\d \theta\right)\d z<\infty,\notag
\end{align}
we finish the proof.
\end{proof}

\end{document}